\documentclass[12pt]{amsart}
\usepackage{amsmath,amsfonts,amsthm,amssymb,indentfirst}
\usepackage{xy} \xyoption{all}

\newtheorem{theorem}{Theorem}[section]
\newtheorem{lemma}[theorem]{Lemma}
\newtheorem{proposition}[theorem]{Proposition}
\newtheorem{corollary}[theorem]{Corollary}

\theoremstyle{definition}
\newtheorem{definition}[theorem]{Definition}
\newtheorem{example}[theorem]{Example}
\newtheorem{notation}[theorem]{Notation}

\theoremstyle{remark}
\newtheorem{remark}[theorem]{Remark}

\numberwithin{equation}{section}

\newcommand{\Spec}{\mathrm{Spec}}
\newcommand{\gr}{\mathrm{gr}}
\newcommand{\gb}{\mathrm{glb}}
\newcommand{\Ind}{\mathrm{Ind}}
\newcommand{\RT}{\mathcal{R}}
\newcommand{\Z}{\mathbb{Z}}
\newcommand{\N}{\mathbb{N}}

\begin{document}

\title{Chains of Semiprime and Prime Ideals  in Leavitt Path Algebras }
\author{G.\ Abrams} 
\address{Department of Mathematics, University of Colorado, Colorado Springs, CO, 80918, USA}
\email{abrams@math.uccs.edu}
\thanks{The first author is partially supported by a Simons Foundation  Collaboration Grants for Mathematicians Award \#208941.}
\author{B.\ Greenfeld}
\address{Department of Mathematics, Bar Ilan University, Ramat Gan 5290002, Israel}
\email{beeri.greenfeld@gmail.com}
\author{Z.\ Mesyan} 
\address{Department of Mathematics, University of Colorado, Colorado Springs, CO, 80918, USA}
\email{zmesyan@uccs.edu}
\author{K.\ M.\ Rangaswamy}
\address{Department of Mathematics, University of Colorado, Colorado Springs, CO, 80918, USA}
\email{krangasw@uccs.edu}



\begin{abstract}
Semiprime ideals of an arbitrary Leavitt path algebra $L$ are described in terms of their generators. This description is then used to show that the semiprime ideals form a complete sublattice of the lattice of ideals of $L$, and they enjoy a certain gap property identified by Kaplansky in prime spectra of commutative rings. It is also shown that the totally ordered sets that can be realized as prime spectra of Leavitt path algebras are precisely those that have greatest lower bounds on subchains and enjoy the aforementioned gap property. Finally, it is shown that a conjecture of Kaplansky regarding identifying von Neumann regular rings via their prime factors holds for Leavitt path algebras.
\end{abstract}

\maketitle

\section{Introduction}

Kaplansky included the following results in two of his monographs from the 1970s, the first of which he attributed to R.\ Hamsher. (See~\cite[page 1]{KapOperator} and~\cite[Theorem 11]{K-1}, respectively.)

\medskip

Theorem (K-1): A commutative ring $R$ is von Neumann regular if and only if $R$ is semiprime and every prime factor ring of $R$  is von Neumann regular.

\medskip

Theorem (K-2): If $P \subsetneq Q$ are prime ideals of a commutative unital ring $R$, then there are prime ideals $P',Q'$ of $R$ such that $P\subseteq P'\subsetneq Q'\subseteq Q$ and $Q'$ covers $P'$, that is, there are no prime ideals of $R$  strictly between $P'$ and $Q'$.

\medskip

Kaplansky took an interest in extending both theorems to noncommutative rings. In each case, his questions inspired substantial new results in the literature. Our goal in this note is to investigate whether Theorems (K-1) and (K-2) hold for Leavitt path algebras. 

In~\cite{KapOperator} Kaplansky conjectured that Theorem (K-1) extends to  all noncommutative rings. While this turned out to not be the case, J.\ Fisher and R.\ Snider~\cite{FS} proved that Kaplansky's conjecture holds precisely when $R$ satisfies the additional condition that the union of any ascending chain of semiprime ideals of $R$ is again a semiprime ideal. (Recall that an ideal $I$ of a ring $R$ is said to be \emph{semiprime} if $I$ is the intersection of a collection of prime ideals of $R$; equivalently, whenever $A^k\subseteq I$  implies $A\subseteq I$, for any ideal $A$ of $R$ and positive integer $k$.) In view of the connection between semiprime ideals and Kaplansky's conjecture indicated in the result of Fisher and Snider, we investigate semiprime ideals in the context of Leavitt path algebras (Section \ref{Semiprimesection}).  We first characterize the semiprime ideals of a Leavitt path algebra $L:=L_K(E)$ for an arbitrary graph $E$ and field $K$. Specifically, we characterize such ideals $I$ in terms of their generators, namely, the vertices inside $I$ and polynomials over $K$, with square-free factors, evaluated at cycles without exits. The semiprime ideals of Leavitt path algebras turn out to be much more well-behaved than the prime ideals. For instance, the semiprime ideals of $L$ form a complete sublattice of the ideal lattice of $L$. As a consequence, we show that the above conjecture of Kaplansky regarding Theorem (K-1) holds for Leavitt path algebras. We also show that every ideal of $L$ is semiprime if and only if the graph $E$ satisfies Condition (K), and that every semiprime ideal of $L$ is prime if and only if the prime ideals of $L$ form a chain under set inclusion.

In \cite{K-1} Kaplansky  asked whether the property of the prime ideals of $R$ in Theorem (K-2), together with the existence of least upper bounds and greatest lower bounds for chains in $\Spec(R)$, the prime spectrum of $R$, characterizes $\Spec(R)$ as a partially ordered set. While this also turned out not to be the case, it opened the door for serious investigations of various aspects of $\Spec(R)$ as a partially ordered set (see, e.g.,~\cite{H, Lewis, LO, S, Speed}). Recently, it was shown in~\cite{AAMS} that a wide class of partially ordered sets can be realized as prime spectra of Leavitt path algebras. In particular, every partially ordered set $(P, \leq)$ that satisfies the property in Theorem (K-2), has greatest lower bounds for chains, and enjoys an additional property (for every downward directed subset $S$ of $P$ and every $p \in P\backslash S$ such that $p$ is greater than the greatest lower bound of $S$, there exists $s \in S$ such that $p > s$) can be realized as the prime spectrum of a Leavitt path algebra. However, for noncommutative rings, it is still an open question whether Theorem (K-2) holds. In Section 4, we continue to investigate Theorem (K-2) in the context of Leavitt path algebras. We say that a ring $R$ satisfies the \emph{Kaplansky Property} (KAP, for short) for a class $\mathcal{F}$ of ideals in $R$, if every pair $P,Q \in \mathcal{F}$ satisfies the condition in Theorem (K-2).  The properties of semiprime ideals mentioned above enable us to show that the Kaplansky Property always holds for semiprime ideals in a Leavitt path algebra. We also show that a totally ordered set can be realized as the prime spectrum of a Leavitt path algebra if and only if the set has greatest lower bounds for all subchains and satisfies the Kaplansky Property.

In the final section (Section \ref{UnionPrimesection}) we consider ideals which are not prime but which are unions of ascending chains of prime ideals. Following~\cite{GRV}, such ideals are called \emph{union-prime ideals}. They are of interest to us since their presence hinders the Kaplansky Property from holding in the prime spectrum of a ring. We first show that the union-prime ideals in a Leavitt path algebra are all graded, and hence always semiprime. To measure the extent to which the union-prime ideals are prevalent in a Leavitt path algebra $L_{K}(E)$, we consider the union-prime index $P^{\Uparrow}$ introduced in~\cite{GRV}, and show that $P^{\Uparrow}(L_{K}(E))\leq|E^{0}|$, where $E^0$ is the set of vertices of the graph $E$. Finally, an example is constructed to show that for every infinite cardinal $\kappa$ there is a directed graph $E_{\kappa}$ such that $P^{\Uparrow}(L_{K}(E_{\kappa}))=\kappa$.

{\bf Acknowledgement.}  We thank the referee for a very thorough reading of the original version.   

\section{Preliminaries}

We give below a short outline of some of the needed basic concepts and results about ordered sets and Leavitt path algebras. For general terminology and facts about Leavitt path algebras, we refer the reader to~\cite{AAS-1, R-1}.

\subsection{Partially ordered sets}

A  \emph{preordered set} $(P, \leq)$ is a set $P$ together with a binary relation $\leq$ which is reflexive and transitive. If in addition $\leq$ is antisymmetric, then $(P,\leq)$ is called a \emph{partially ordered set} (often shortened to \emph{poset}). If a poset $(P, \leq)$ has the property that $p\leq q$ or $q \leq p$ for all $p,q \in P$, then $(P, \leq)$ is called a  \emph{totally ordered set}, or a \emph{chain}.  

Let $(P, \leq)$ be a partially ordered set. An element $x$ of $P$ is called a \emph{least} element of $P$ if $x \leq y$ for all $y \in P$, and $x$ is a \emph{greatest} element of $P$ if $y \leq x$ for all $y \in P$. $P$ is \emph{downward directed} if it is nonempty, and for all $p,q \in P$ there exists $r \in P$ such that $p \geq r$ and $q\geq r$.

For any poset $(P, \leq)$ and subset $S$ of $P$, by restriction $(S, \leq)$ is a poset. A \emph{lower bound} for $S \subseteq P$ is an element $x$ of $P$ such that $x\leq s$ for all $s\in S$. A \emph{greatest lower bound} of $S$ is a lower bound $x$ of $S$ with the additional property that $y\leq x$ for every $y\in P$ having $y\leq s$ for all $s\in S$. \emph{(Least) upper bounds} are defined analogously.
  
The set of the integers will be denoted by $\Z$, the set of the natural numbers (including $0$) will be denoted by $\N$, and the cardinality of a set $X$ will be denoted by $|X|$.  

\subsection{Directed Graphs}

A \emph{(directed) graph} $E=(E^{0},E^{1},r,s)$ consists of two sets $E^{0}$ and
$E^{1}$ together with maps $r,s:E^{1}\rightarrow E^{0}$. The elements of
$E^{0}$ are called \emph{vertices} and the elements of $E^{1}$ \emph{edges}. All the graphs $E$ that we consider (excepting when specifically stated) are arbitrary in the sense that no restriction is placed either on the number of vertices in $E$ or on the number of edges emitted by a single vertex.

A vertex $v \in E^0$ is called a \emph{sink} if it emits no edges, and $v$ is called a \emph{regular vertex} if it emits a nonempty finite set of edges. An \emph{infinite emitter} is a vertex which emits infinitely many edges. A graph without infinite emitters is said to be \emph{row-finite}. For each $e\in E^{1}$, we call $e^*$ a \emph{ghost edge}. We let $r(e^*)$ denote $s(e)$, and we let $s(e^*)$ denote $r(e)$. A \emph{path} $\mu$ of length $n>0$, denoted $|\mu | = n$,  is a finite sequence of edges $\mu=e_{1}e_{2}\cdots e_{n}$ with $r(e_{i})=s(e_{i+1})$ for all $i=1,\dots, n-1$. In this case $\mu^*=e_{n}^*\cdots e_{2}^*e_{1}^*$ is the corresponding \emph{ghost path}. A vertex is considered to be a path of length $0$. The set of all vertices on a path $\mu$ is denoted by $\mu^{0}$.

A path $\mu =e_1\cdots e_{n}$ in $E$ is \emph{closed} if $r(e_{n})=s(e_{1})$, in which case $\mu$ is said to be \emph{based} at the vertex $s(e_{1})$. A closed path $\mu$ as above is called \emph{simple} provided it does not pass through its base more than once, i.e., $s(e_{i})\neq s(e_{1})$ for all $i=2,\dots,n$. The closed path $\mu$ is called a \emph{cycle} if it does not pass through any of its vertices twice, that is, if $s(e_{i})\neq s(e_{j})$ for every $i\neq j$. A cycle consisting of just one edge is called a \emph{loop}. An \emph{exit} for a path $\mu =e_{1} \cdots e_{n}$ is an edge $f$ that satisfies $s(f)=s(e_{i})$ for some $i$, but where $f\neq e_{i}$. The graph $E$ is said to satisfy \emph{Condition (L)} if every cycle in $E$ has an exit in $E$. The graph $E$ is said to satisfy \emph{Condition (K)} if any vertex on a simple closed path $\mu$ is also the base for a simple closed path $\gamma$ different from $\mu$. If $E$ has no cycles, then it is called \emph{acyclic}.

If there is a path in $E$ from a vertex $u$ to a vertex $v$, then we write $u\geq v$. It is easy to see that $(E^0, \leq)$ is a preordered set. A nonempty subset $D$ of $E^0$ is said to be \emph{downward directed} if for any $u,v\in D$, there exists a $w\in D$ such that $u\geq w$ and $v\geq w$. A subset $H$ of $E^{0}$ is called \emph{hereditary} if, whenever $v\in H$ and $w\in E^{0}$ satisfy $v\geq w$, then $w\in H$. A  set $H \subseteq E^0$ is \emph{saturated} if $r(s^{-1}(v))\subseteq H$ implies that $v\in H$, for any regular vertex $v$.

\subsection{Leavitt Path Algebras}\label{LPA-subsection}

Given a (nonempty) graph $E$ and a field $K$, the \emph{Leavitt path algebra} $L_{K}(E)$ is defined to be the $K$-algebra generated by $\{v:v\in
E^{0}\} \cup \{e,e^*:e\in E^{1}\}$, subject to the following relations:

\smallskip

{(V)} \ \ \ \ \ $vv = v$ and $vw = 0$ if $v \neq w$, for all $v,w\in E^0$;

{(E1)} \ \ \ \ $s(e)e=e=er(e)$ for all $e\in E^{1}$;

{(E2)} \ \ \ \ $r(e)e^*=e^*=e^*s(e)$ for all $e\in E^{1}$;

{(CK-1)} \ $e^*e=r(e)$ and $e^*f=0$ if $e\neq f$, for all $e,f\in E^{1}$;

{(CK-2)} \ $v=\sum_{e\in s^{-1}(v)}ee^*$ for every regular $v\in E^{0}$.

\smallskip

We note that every element of $L_K(E)$ can be expressed in the form $\sum_{i=1}^n k_i\alpha_i\beta_i^*$ for some $k_i \in K$ and paths $\alpha_i,\beta_i$ in $E$, and that sums of distinct vertices in $E^0$ form a set of local units for $L_K(E)$. Moreover, every Leavitt path algebra $L_{K}(E)$ is $\Z$-graded. Specifically, $L_{K}(E)=\bigoplus_{n\in\Z}L_{n}$, where 
\[L_{n}=\bigg\{\sum_i k_{i}\alpha_{i}\beta_{i}^*\in L:\text{ }|\alpha_{i}|-|\beta_{i}|=n\bigg\}.\] 
This grading is induced by defining, for all $v\in E^{0}$ and $e\in E^{1}$, $\deg (v)=0$, $\deg(e)=1$, $\deg(e^*)=-1$. Here the $L_{n}$ are abelian subgroups satisfying $L_{m}L_{n}\subseteq L_{m+n}$ for all $m,n\in \mathbb{Z}$.  An ideal $I$ of $L_{K}(E)$ is said to be a \emph{graded ideal} if $I = \bigoplus_{n\in\mathbb{Z}} (I\cap L_{n})$. Equivalently, if $a\in I$ and $a=a_{i_{1}}+\cdots+a_{i_{m}}$ is a graded sum with $a_{i_{k}}\in L_{i_{k}}$ for all $k=1,\dots,m$, then $a_{i_{k}}\in I$ for all $k$.

We shall be using some of the results from~\cite{AA, AAS, T}, where the graphs that the authors consider are assumed to be countable. We point out, however, that the results in these papers hold for arbitrary graphs, with no restriction on the number of vertices or the number of edges emitted by any vertex. In fact, these results without any restriction on the size of the graph are proved in~\cite{AAS-1}. For the convenience of the reader, we shall still give references to~\cite{AA, AAS, T} in addition to~\cite{AAS-1}. 

Given a graph $E$, a \emph{breaking vertex} of a hereditary saturated subset $H \subseteq E^0$ is an infinite emitter $w\in E^{0}\backslash H$ with the property that $0<|s^{-1}(w)\cap r^{-1}(E^{0}\backslash H)|<\infty$. The set of all breaking vertices of $H$ is denoted by $B_{H}$. For any $v\in B_{H}$, $v^{H}$ denotes the element $v-\sum_{s(e)=v, \, r(e) \notin H}ee^*$. Given a hereditary saturated subset $H$ and $S\subseteq B_{H}$, $(H,S)$ is called an \emph{admissible pair}, and the ideal of $L_K(E)$ generated by $H\cup \{v^{H}:v\in S\}$ is denoted by $I(H,S)$. It was shown in~\cite[Theorem 5.7]{T} (see also~\cite[Theorem 2.5.8]{AAS-1}) that the graded ideals of $L_{K}(E)$ are precisely the ideals of the form $I(H,S)$. Moreover, given two admissible pairs $(H_1,S_1)$ and $(H_2,S_2)$, setting $(H_1,S_1) \leq (H_2,S_2)$ whenever $H_1 \subseteq H_2$ and $S_1 \subseteq H_2 \cup S_2$, defines a partial order on the set of all admissible pairs of $L_K(E)$. The map $(H,S) \mapsto I(H,S)$ gives a one-to-one order-preserving correspondence between the poset of admissible pairs and the set of all graded ideals of $L_K(E)$, ordered by inclusion. In~\cite[Theorem 5.7]{T} (see also~\cite[Theorem 2.4.15]{AAS-1}) it was also shown that $L_{K}(E)/I(H,S)\cong L_{K} (E\backslash(H,S))$ for any admissible pair $(H,S)$. Here $E\backslash(H,S)$ is a \emph{quotient graph of} $E$ where \[(E\backslash(H,S))^{0}=(E^{0}\backslash H)\cup\{v':v\in B_{H} \backslash S\}\]
and
\[(E\backslash(H,S))^{1}=\{e\in E^{1}:r(e)\notin H\}\cup\{e':e\in E^{1} \text{ with } r(e)\in B_{H}\backslash S\},\]
and $r,s$ are extended to $(E\backslash(H,S))^{0}$ by setting $s(e^{\prime
})=s(e)$ and $r(e')=r(e)'$.

We conclude this section by recalling the following three results, since we shall refer to them frequently in our investigation.

\begin{proposition}[well-known; see, e.g., Proposition 3 in \cite{AAMS}] \label{IntersectionProperty}
Let $R$ be any ring, and let $\, \{P_\lambda :  \lambda \in \Lambda\}$ be a downward directed set (under set inclusion) of prime ideals of  $R$. Then $I = \bigcap_{\lambda \in \Lambda} P_\lambda$ is a prime ideal of $R$.  
\end{proposition}

\begin{theorem}[Theorem 4 in \cite{R-2}] \label{non-graded-threorem}
Let $E$ be a graph, $K$ a field, and $L:=L_{K}(E)$. Also, let $I$ be a non-graded ideal of $L$, and set $H=I\cap E^{0}$ and $S=\{u\in B_{H}:u^{H}\in I\}$. Then $I=I(H,S)+\sum_{i\in X} \langle f_i(c_i) \rangle$ where $X$ is an index set, each $c_i$ is a cycle without exits in $E\backslash(H,S)$, and each $f_i(x)\in K[x]$ is a polynomial with a nonzero constant term.

Here $I(H,S)$ is called the \emph{graded part} of the ideal $I$, and is denoted by $\gr(I)$.
\end{theorem}

In a graph $E$, a cycle $c = e_1e_2 \cdots e_n$ is said to be \emph{WK} (for ``without Condition (K)") if no vertex along $c$ is the source of another distinct cycle in $E$.   (In this context we view $c$ and any ``shift" $e_i \cdots e_n e_1 \cdots e_{i-1}$ of $c$ as being the same cycle.)  

\begin{theorem}[Theorem 3.12 in \cite{R-1}] \label{primeclass}
Let $E$ be a graph, $K$ a field, $I$ a proper ideal of $L_K(E)$, and $H = I \cap E^0$. Then $I$ is a prime ideal if and only if $I$ satisfies one of the following conditions.
\begin{enumerate}
\item[$(1)$] $I = \langle H \cup \{v^H : v \in B_H\} \rangle$ and $E^0\setminus H$ is downward directed. 
\item[$(2)$] $I = \langle H \cup \{v^H : v \in B_H\setminus \{u\}\}\rangle$ for some $u \in B_H$ and $E^0\setminus H = \{v \in E^0 : v \geq u\}$.
\item[$(3)$] $I = \langle H \cup \{v^H : v \in B_H\} \cup \{f(c)\}\rangle$ where $c$ is a WK cycle having source $u \in E^0$, $E^0\setminus H = \{v \in E^0 : v \geq u\}$, and $f(x)$ is an irreducible polynomial in $K[x, x^{-1}]$.
\end{enumerate}
\end{theorem}

\section{Semiprime Ideals}\label{Semiprimesection}

In this section, we give a complete description of the semiprime ideals in a Leavitt path algebra $L:=L_{K}(E)$, of an arbitrary graph $E$ over a field $K$, in terms their generators. We show that the semiprime ideals of $L$ form a complete sublattice of the lattice of two-sided ideals of $L$. In particular, the union of an ascending chain of semiprime ideals is always semiprime. We also describe when every ideal in $L$ is semiprime and when every semiprime ideal is prime.

\begin{remark} \label{ideals in K[x,x^-1]} 
We note that every nonzero ideal $I$ in the principal ideal domain $K[x,x^{-1}]$ is generated by a polynomial $f(x)\in K[x]$ with a nonzero constant term. For, suppose that $I=\langle g(x,x^{-1}) \rangle$, where $g(x,x^{-1})=\sum_{i=m}^{n} a_{i}x^{i}$ for some $m \leq n$ ($m,n\in \Z)$ and $a_i \in K$, with $a_{m}\neq 0 \neq a_{n}$. Then 
\[g(x,x^{-1})=x^{m}\sum_{i=0}^{n-m} b_{i}x^{i}=  x^m f(x),\] 
where $b_{i}=a_{i+m}$ and $b_{0}=a_{m}\neq0$. Since $x^{m}$ is a unit in $K[x,x^{-1}]$, we have $I=\langle f(x) \rangle$, where $f(x)$ is a polynomial in $K[x]$ with a nonzero constant term. 

Recall that an element $r$ of a commutative unital ring $R$ is called {\it irreducible} if $r$ is a nonunit, and $r$ cannot be written as $r = st$ with both $s,t$ nonunits in $R$.  It  follows from the previous observation (and the fact that the units in $K[x,x^{-1}]$ are of the form $kx^m$ for $ k \in K \setminus \{0\}$ and $m\in \Z$) that the irreducible polynomials in $K[x,x^{-1}]$ are precisely those of the form $x^m f(x)$ for some $m\in \Z$, where $f(x)$ is an irreducible polynomial in  $K[x]$.  
\end{remark}

We record a well-known fact  which  will be useful  in the proof Theorem~\ref{characterization}.

\begin{lemma} \label{semiprime in LaurentPolyRing} 
Let $K$ be a field and $A$ an  ideal of $K[x,x^{-1}]$, and write $A=\langle f(x)\rangle$ for some $f(x) \in K[x,x^{-1}]$. Then $A$ is semiprime if and only if $f(x)$ is a product of distinct irreducible polynomials.
\end{lemma}

\begin{proof}
If $f(x)=p_{1}(x)\cdots p_{n}(x)$ is a product of distinct
irreducible polynomials, then it is easy to see that $\langle f(x)\rangle =\langle p_{1}(x) \rangle\cap\cdots\cap\langle p_{n}(x)\rangle$, and since each $\langle p_{i}(x)\rangle$ is a prime ideal, $A$ is semiprime.

Conversely, suppose that $A = \langle f(x) \rangle$, where 
$f(x)=p(x)^{k}g(x)$,  $p(x)$ is an irreducible polynomial, $k\geq 2$,  and
$\gcd(p(x),g(x))=1$.   Then $\langle p(x)g(x) \rangle ^k \subseteq A$, but $\langle p(x)g(x) \rangle \nsubseteq A$, so that $A$ is not semiprime.
\end{proof}

Next, observe that any graded ideal $I(H,S)$ of a Leavitt path algebra $L$ is semiprime, since $L/I(H,S)\cong L_{K}(E\backslash(H,S))$ is a Leavitt path algebra, and hence has zero prime radical. (It is well-known that the Jacobson radical, and hence also the prime radical, of any Leavitt path algebra is zero; see~\cite[Proposition 6.3]{AA} or~\cite[Proposition 2.3]{AAS-1}.) But not all semiprime ideals need to be graded ideals. For example, let $E$ be the graph having one vertex $v$ and one loop $e$, with $s(e)=v=r(e)$, pictured below. 

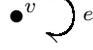
\[\xymatrix{{\bullet}^{v} \ar@(ur,dr)^e}\]
\medskip

\noindent Now $K[x,x^{-1}]\cong L_{K}(E)$, via the map induced by sending $1\mapsto v$, $x\mapsto e$,  and $x^{-1}\mapsto e^*$. Then the natural $\mathbb{Z}$-grading of $K[x,x^{-1}]$ induced by integral powers of $x$ coincides with the grading of $L_{K}(E)$ described above. Under this grading $\{0\}$ is the only proper graded ideal of $K[x,x^{-1}]$, since integral powers of $x$ are units in $K[x,x^{-1}]$, and hence every nonzero proper ideal of $K[x,x^{-1}]$, and particularly every nonzero semiprime ideal, is non-graded.

The next result gives a complete description of all the semiprime ideals of
$L_K(E)$ in terms of their generators.

\begin{theorem} \label{characterization}
Let $E$ be a graph, $K$ a field, and $L:=L_{K}(E)$. Also let $A$ be an ideal of $L$, with $A\cap E^{0}=H$ and $S=\{v\in B_{H}:v^{H}\in A\}$. Then the following are equivalent.
\begin{enumerate}
\item $A$ is a semiprime ideal.

\item $A=I(H,S)+\sum_{i\in X} \langle f_{i}(c_{i}) \rangle$ where $X$ is a possibly empty index set, each $c_{i}$ is a cycle without exits in $E\backslash(H,S)$, and each $f_{i}(x)$ is a polynomial with nonzero constant term, which is a product of distinct irreducible polynomials in $K[x,x^{-1}]$.
\end{enumerate}
\end{theorem}

\begin{proof}
We first observe that if $A$ is a non-graded ideal, then by Theorem~\ref{non-graded-threorem}, $A=I(H,S)+\sum_{i\in X} \langle f_{i}(c_{i}) \rangle$, where $X$ is a nonempty index set, each $f_{i}(x)\in K[x]$ has nonzero constant term, and each $c_{i}$ is a cycle without exits in $E\backslash(H,S)$, based at a vertex $v_i$. Moreover, we may assume that $c_i^0 \neq c_j^0$ for $i\neq j$. In $L/I(H,S)$, as noted in~\cite[Proposition 3.5]{AAS} (see also~\cite[Theorem 2.7.3]{AAS-1}), the ideal $M$ generated by $\{c_{i}^{0}:i\in X\}$ decomposes as $M=\bigoplus_{i\in X} M_{i}$, where $M_{i}= \langle c_{i}^{0} \rangle \cong M_{\Lambda_{i}}(K[x,x^{-1}])$ is the ring of $\Lambda_{i}\times\Lambda_{i}$ matrices with at most finitely many nonzero entries, and $\Lambda_{i}$ is the set of all paths that end at $v_{i}$ but do not include the entire $c_{i}$. Moreover, given two paths $r,p \in \Lambda_{i}$, the isomorphism $\theta_i : M_{i} \to M_{\Lambda_{i}}(K[x,x^{-1}])$ takes $rc_i^kp^*$ to $x^ke_{rp}$ for all $k \in \Z$, where $e_{rp}$ is a matrix unit. Clearly \[A/I(H,S)=\bigoplus_{i\in X} \langle f_{i}(c_{i}) \rangle \subseteq \bigoplus_{i\in X} M_{i}.\]

Assume that (1) holds. If $A$ is a graded ideal, then $A=I(H,S)$, by the remarks in Section~\ref{LPA-subsection}, and we are done. Let us therefore suppose that $A$ is a non-graded ideal, and hence $A=I(H,S)+\sum_{i\in X} \langle f_{i}(c_{i}) \rangle$, with $X$, $f_i$, and $c_i$ as in the preceding paragraph. Also, as before, \[A/I(H,S)=\bigoplus_{i\in X}  \langle f_{i}(c_{i}) \rangle\subseteq \bigoplus_{i\in X} M_{i}=M.\] As $A$ is a semiprime ideal, $A/I(H,S)$ is semiprime in $L/I(H,S)$. Since $M$ is a graded ideal, and hence a ring with local units (namely, of sums of distinct vertices of $E^0$ in $M$), it is easy to see that the intersection of any prime ideal of $L/I(H,S)$ with $M$ is again a prime ideal of $M$. Hence $A/I(H,S)$ is also a semiprime ideal of $M$, which implies that $\langle f_{i}(c_{i}) \rangle$ is a semiprime ideal of $M_{i}$. Now
\[M_{i}\overset{\theta_i}{\cong} M_{\Lambda_{i}}(K[x,x^{-1}]),\] 
and it is straightforward to show that  the ideals of $M_{\Lambda_{i}}(K[x,x^{-1}])$ are of the form $M_{\Lambda_{i}}(\langle f(x)\rangle)$ for various $f(x) \in K[x,x^{-1}]$. From the description of the isomorphism $\theta_i$ above, it follows that the ideal $\langle f_{i}(c_{i})\rangle$ of $M_{i}$ corresponds to the ideal $M_{\Lambda_{i}}(\langle f_{i}(x)\rangle)$. Since $M_{\Lambda_{i}}(K[x,x^{-1}])$ is Morita equivalent to the ring $K[x,x^{-1}]$, the ideal lattices of these rings are isomorphic, and under this isomorphism prime and semiprime ideals of $M_{\Lambda_{i}}(K[x,x^{-1}])$ correspond to prime and semiprime ideals, respectively, of $K[x,x^{-1}]$ (see e.g.,  \cite[Propositions 3.3 and 3.5]{AM}). Since $\langle f_{i}(c_{i})\rangle$ is a semiprime ideal of $M_{i}$, we see that $\langle f_{i}(x)\rangle$ is a semiprime ideal of $K[x,x^{-1}]$. By Lemma \ref{semiprime in LaurentPolyRing}, we then conclude that $f_{i}(x)$ is a product of distinct irreducible polynomials in $K[x,x^{-1}]$, proving (2).

Conversely, assume that (2) holds. Thus $A=I(H,S)+\sum_{i\in X} \langle f_{i}(c_{i}) \rangle$, where we may assume that $X$ is nonempty, the stated condition holds for the $c_{i}$, and each $f_{i} (x)$ is a product of distinct irreducible polynomials in $K[x,x^{-1}]$. As noted in the first paragraph of the proof, we have 
\[A/I(H,S)=\bigoplus\limits_{i\in X} \langle f_{i}(c_{i})\rangle \subseteq \bigoplus_{i\in X} M_{i}=M,\] 
where $M_{i}\cong M_{\Lambda_{i}}(K[x,x^{-1}])$ for each $i \in X$. Now, by Lemma~\ref{semiprime in LaurentPolyRing}, $\langle f_{i}(x)\rangle$ is a semiprime ideal of $K[x,x^{-1}]$ and, as noted earlier, the Morita equivalence of $K[x,x^{-1}]$ with $M_{\Lambda_{i}}(K[x,x^{-1}])$ then implies that $\langle f_{i}(c_{i})\rangle$ is semiprime in $M_{i}$. Since this is true for all $i\in X$, and since we may view $\bigoplus_{i\in X} M_{i}$ as a direct sum of rings, we conclude that $A/I(H,S)$ is a semiprime ideal of $\bigoplus_{i\in X} M_{i}=M$. Now consider the exact sequence \[0\rightarrow M/\bar{A} \rightarrow\bar{L}/\bar{A}\rightarrow\bar{L}/M\rightarrow 0,\] where $\bar{A}=A/I(H,S)$ and $\bar{L}=L/I(H,S)$. The graded ideal $M$ is semiprime in $\bar{L}$, and hence $\bar{L}/M$ is a semiprime ring. Since $M/\bar{A}$ is also a semiprime ring, we conclude that $\bar{L}/\bar{A}$ is semiprime as well. Finally, since $L/A\cong\bar{L}/\bar{A}$, we see that $A$ is a semiprime ideal of $L$, proving (1).
\end{proof}

Theorem~\ref{characterization} has the following consequence; in the context of results to be presented in the sequel, this property of semiprime ideals is somewhat unexpected.

\begin{proposition} \label{semiprimeLattice} 
Let $E$ be a graph, $K$ a field, and $L:=L_{K}(E)$. Then the semiprime ideals of $L$ form a complete sublattice $S$ of the lattice of ideals of $L$, that is, $S$ is closed under taking arbitrary sums and intersections.

In particular, the union of a chain of semiprime ideals of $L$ is again a semiprime ideal.
\end{proposition}

\begin{proof}
Suppose that $\{A_{t}:t\in T\}$ is a set of semiprime ideals of $L$. Since each $A_{t}$ is an intersection of prime ideals, clearly $\bigcap_{t\in T} A_{t}$ is a semiprime ideal of $L$, showing that $S$ is closed under taking intersections.

Setting $H_{t} = A_{t}\cap E^{0}$ for each $t \in T$, by Theorem~\ref{characterization}, $A_{t}=\gr(A_{t})+\sum_{i_{t}\in X_{t}} \langle f_{i_{t}}(c_{i_{t}})\rangle$, where each $c_{i_{t}}$ is a cycle without exits in $E^{0}\backslash H_{i_{t}}$, based at a vertex $v_{i_{t}}$, and each $f_{i_{t}}(x)$ is a product of distinct irreducible polynomials in $K[x,x^{-1}]$. Let $A=\sum_{t\in T} A_{t}$ and $H = A\cap E^{0}$. Since a sum of graded ideals is again a graded ideal, $\sum_{t\in T} \gr(A_{t})\subseteq \gr(A)$. It is possible that $\gr(A)$ contains some of the vertices $v_{i_{t}}$ and hence the ideals $\langle f_{i_{t}}(c_{i_{t}})\rangle$. Removing such ideals from the sum, we can write $A=\gr(A)+ \sum_{i\in I} \langle f_{i}(c_{i})\rangle$, where each $c_{i}$ is a cycle without exits in $E^{0}\backslash H$, and each $f_{i}(x)$ is a product of distinct irreducible polynomials in $K[x,x^{-1}]$. Therefore, by Theorem \ref{characterization}, $A$ is a semiprime ideal of $L$, showing that $S$ is closed under taking sums.

The final claim follows from the fact that the union of a chain of ideals is equal to its sum.
\end{proof}

We conclude this section by describing the graphs $E$ for which every ideal of $L_K(E)$ is semiprime, and the ones for which every semiprime ideal of $L_K(E)$ is prime.

\begin{proposition} \label{All semiprime} 
Let $E$ be a graph, $K$ a field, and $L:=L_{K}(E)$. Then the following are equivalent.
\begin{enumerate}
\item Every ideal of $L$ is semiprime.

\item Every ideal of $L$ is graded.

\item The graph $E$ satisfies Condition {\rm (K)}.
\end{enumerate}
\end{proposition}

\begin{proof}
Assume that (1) holds. Suppose also that there is a non-graded ideal $I$ of $L$, and let $H=I\cap E^{0}$. By Theorem~\ref{non-graded-threorem}, $I=\gr(I)+\sum_{i\in X}\langle f_{i}(c_{i})\rangle$, where each $c_{i}$ is a cycle without exits in $E^{0}\backslash H$, based at a vertex $v_{i}$, and each $f_{i}(x)\in K[x]$. Then for any $i \in X$, the ideal $A=\gr(I)+\langle (v_{i}+c_{i})^{2}\rangle$ is not semiprime, since the ideal $B=\gr(I)+\langle v_{i}+c_{i}\rangle$ satisfies $B^{2} \subseteq A$ and $B\not\subseteq A$. This contradicts (1), and hence every ideal of $L$ must be graded, proving (2).

Now (2) implies (1), as noted after Lemma~\ref{semiprime in LaurentPolyRing}, and the equivalence of (2) and (3) is well-known (see~\cite[Theorem 6.16]{T} or~\cite[Theorem 3.3.11]{AAS-1}).
\end{proof}

\begin{proposition} \label{semiprimePrime} 
Let $E$ be a graph, $K$ a field, and $L:=L_{K}(E)$. Then the following are equivalent.
\begin{enumerate}
\item Every semiprime ideal of $L$ is a prime ideal.

\item The prime ideals of $L$ form a chain under set inclusion.

\item Every ideal of $L$ is both  graded and prime, and the ideals of $L$ form a chain under set inclusion.

\item The graph $E$ satisfies Condition {\rm (K)}, and the admissible pairs of $E$ form a
chain under the partial order of the admissible pairs.
\end{enumerate}
\end{proposition}

\begin{proof}
Assume that (1) holds. To prove (2), suppose, on the contrary, that there are two prime ideals $A,B$ of $L$ such that $A\nsubseteq B$ and $B\nsubseteq A$. Choose an element $a\in A\backslash B$ and an element $b\in B\backslash A$. Since $aLb$ is contained in both $A$ and $B$, we have $aLb\subseteq A\cap B$. Now $A\cap B$ is a semiprime ideal and is hence prime, by supposition. Thus one of $a$ and $b$ must belong to $A\cap B$. This means that either $a\in B$ or $b\in A$; a contradiction. Hence the prime ideals of $L$ form a chain under set inclusion, proving (2).

Assume that (2) holds. We now claim that every ideal of $L$ is graded. Suppose that there is a non-graded ideal $A$ in $L$, and set $H = A\cap E^{0}$. Now $\gr(A)$ is a semiprime ideal and hence must be prime, since by hypothesis it is the intersection of a chain of prime ideals, so that Proposition \ref{IntersectionProperty} applies. This means that $E^{0}\backslash H$ is downward directed, by Theorem~\ref{primeclass}, and hence $E^{0}\backslash H$ can contain at most one cycle without exits. Thus, by Theorem~\ref{non-graded-threorem}, $A$ must be of the form $A=\gr(A)+\langle f(c)\rangle$, where $c$ is the unique cycle without exits in $E^{0}\backslash H$ and $f(x)\in K[x]$. Now for any two distinct irreducible $p(x), q(x) \in K[x]$ with nonzero constant terms, the ideals $P=\gr(A)+\langle p(c)\rangle$ and $Q=\gr(A)+\langle q(c)\rangle$ are prime, by Theorem~\ref{primeclass}, with the property that  $P\nsubseteq Q$ and $Q\nsubseteq P$, contradicting (2). Thus every ideal $I$ of $L$ must be graded, and hence also semiprime.   Since the prime ideals of $L$ form a chain, intersections of prime ideals of $L$ are also prime (again using Proposition \ref{IntersectionProperty}). Therefore every ideal of $L$ is prime, proving (3).

Assume that (3) holds. Then $E$ satisfies Condition (K), by Proposition~\ref{All semiprime}, and the admissible pairs of $E$ form a chain, by the comments in Section~\ref{LPA-subsection}, proving (4).

Finally, if (4) holds, then every ideal of $L$ is graded, by Proposition~\ref{All semiprime}. Since admissible pairs correspond to generating sets for graded ideals, this means that the ideals of $L$ form a chain. Since semiprime ideals are intersections of prime ideals, and the intersection of any chain of prime ideals is prime, (1) follows.
\end{proof}

\section{The Kaplansky Conjecture and the Kaplansky Property}

In this section the Kaplansky conjecture regarding Theorem (K-1), mentioned in the Introduction, is shown to be true for Leavitt path algebras $L$. We then show that the Kaplansky Property holds for semiprime ideals of $L$, and give sufficient conditions under which prime ideals of $L$ satisfy this property. In particular, we use the Kaplansky Property to characterize the totally ordered sets that can be realized as prime spectra of Leavitt path algebras.

Our first result shows that Kaplansky's conjecture holds in Leavitt path algebras, and it is stated in a slightly stronger form, in the sense that we do not need the additional hypothesis that the ring is semiprime (as all Leavitt path algebras are semiprime), and that we only need to consider graded prime ideals. 

Recall that a ring $R$ is called \emph{von Neumann regular} if for every $r \in R$ there exists a $p \in R$ such that $r=rpr$.

\begin{proposition} \label{KapConj} 
Let $E$ be a graph, $K$ a field, and $L:=L_{K}(E)$. Then $L$ is von Neumann regular if and only if $L/P$ is von Neumann regular for any graded prime ideal $P$ of $L$.
\end{proposition}

\begin{proof}
Assume that $L/P$ is von Neumann regular for all graded prime ideals $P$. We claim that $E$ contains no cycles. Suppose, on the contrary, that there is a cycle $c$, based at a vertex $v$, in the graph $E$, and let $H=\{u\in E^{0}:u\ngeq v\}$. Clearly $H$ is a hereditary saturated set of vertices and $E^{0}\backslash H$ is downward directed. Now $P=I(H,B_{H})$ is a prime ideal, by Theorem~\ref{primeclass}, and $L/P\cong L_{K}(E\backslash(H,B_{H}))$, by the remarks in Section~\ref{LPA-subsection}. By hypothesis, $L_{K} (E\backslash(H,B_{H}))$ is von Neumann regular, and by~\cite[Theorem 1]{AR}, $E\backslash(H,B_{H})$ must then be acyclic, contradicting $c$ being a cycle in $E\backslash(H,B_{H})$. Hence $E$ must be acyclic and we appeal to~\cite[Theorem 1]{AR} to conclude that $L$ is von Neumann regular.

The converse follows from the fact that any quotient of a von Neumann regular ring is again von Neumann regular.
\end{proof}

\begin{remark}
As mentioned in the Introduction, Fisher and Snider showed in~\cite[Theorem 1.1]{FS} that a semiprime ring $R$ with all its prime factor rings regular is von Neumann regular, provided the union of any ascending chain of semiprime ideals of $R$ is again a semiprime ideal. Since a Leavitt path algebra $L$ is always semiprime, and since Proposition~\ref{semiprimeLattice} shows that the union of an ascending chain of semiprime ideals in $L$ is semiprime, we can appeal to~\cite[Theorem 1.1]{FS} to get an alternate proof of Proposition~\ref{KapConj}.
\end{remark}

Next we explore the Kaplansky Property (KAP). The following  proposition gives
conditions for (KAP) to hold for a given family of ideals in a ring $R$. The
argument used here is fundamentally due to Kaplansky~\cite[Theorem 11]{K-1}. For convenience, we say that a pair of ideals $(P,Q)$ in a ring is a \emph{Kaplansky pair}, if the ideals $P,Q$ satisfy the covering property stated in Theorem (K-2).

\begin{proposition}
\label{KAP} Let $R$ be an arbitrary ring. Suppose that $\mathcal{F}$ is a nonempty family of ideals of $R$ such that $\mathcal{F}$ is closed under taking unions and intersections of chains of ideals. Then {\rm (KAP)} holds for ideals in $\mathcal{F}$.
\end{proposition}

\begin{proof}
Suppose that $A\subsetneq B$ are two elements of $\mathcal{F}$, and let $x\in B\backslash A$. Since the union of a chain of ideals in $\mathcal{F}$ is again a member of $\mathcal{F}$, we apply Zorn's Lemma to obtain an ideal $M\in \mathcal{F}$ maximal with respect to the property that $A\subseteq M\subsetneq B$ and $x\notin M$. Let $C=M+\langle x\rangle$. Since the intersection of a chain of ideals in $\mathcal{F}$ is a member of $\mathcal{F}$, we can similarly apply Zorn's Lemma to obtain an ideal $N \in \mathcal{F}$ minimal with respect to the property that $C \subseteq N \subseteq B$. (Note that if $C \in \mathcal{F}$, then $C = N$.) We claim that $(M,N)$ is a Kaplansky pair. For, suppose that there is an ideal $P\in \mathcal{F}$ that satisfies $M\subsetneq P\subseteq N$. Then $x\in P$,  by the maximality of $M$, and hence $C\subseteq P$. The minimality of $N$ then implies that $P=N$. Thus (KAP) holds for ideals belonging to $\mathcal{F}$.
\end{proof}

\begin{corollary} \label{KAP for semiprimes} 
Let $E$ be a graph, $K$ a field, and $L:=L_{K}(E)$. Then the semiprime ideals of $L$ satisfy {\rm (KAP)}.
\end{corollary}

\begin{proof}
By Proposition~\ref{semiprimeLattice}, the set of semiprime ideals of $L$ is closed under taking unions and intersections of chains. We then appeal to Proposition~\ref{KAP} to conclude that (KAP) holds for semiprime ideals of $L$.
\end{proof}

Since the intersection of a chain of prime ideals in any ring is again a prime ideal (Proposition \ref{IntersectionProperty}), the following corollary follows immediately from Proposition~\ref{KAP}.

\begin{corollary} \label{KAP for Primes} 
Let $R$ be an arbitrary ring. If the union of any chain of prime ideals in $R$ is again a prime ideal, then {\rm (KAP)} holds for the prime ideals of $R$.
\end{corollary}

The sufficient condition in Corollary \ref{KAP for Primes} is not necessary, as the next example shows.

\begin{example}
\label{row-finite Union prime} Consider the following row-finite graph $E$.

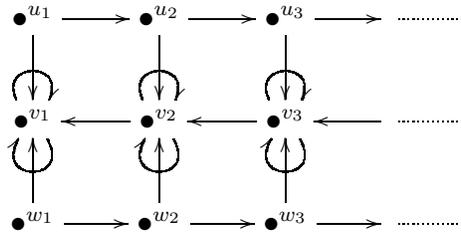
\[\xymatrix{ 
{\bullet}^{u_{1}} \ar [r]  \ar[d] & {\bullet}^{u_{2}} \ar [r] \ar[d] & {\bullet}^{u_{3}} \ar [r]  \ar[d] & \ar@{.}[r] & \\
{\bullet}^{v_{1}} \ar@(ul,ur) \ar@(dr,dl) & {\bullet}^{v_{2}} \ar@(ul,ur) \ar@(dr,dl) \ar [l]  & {\bullet}^{v_{3}} \ar@(ul,ur) \ar@(dr,dl) \ar [l] & \ar [l]  \ar@{.}[r] &\\
{\bullet}^{w_{1}} \ar [r] \ar[u] & {\bullet}^{w_{2}} \ar [r] \ar[u]  & {\bullet}^{w_{3}} \ar [r] \ar[u] & \ar@{.}[r] &\\
}\]
\medskip

\noindent We claim that (KAP) holds for the prime ideals of $L:=L_{K}(E)$. To see this, first observe that, since $E$ satisfies Condition (K), all the ideals of $L$ are necessarily graded, by Proposition~\ref{All semiprime}. In addition, $E$ is row-finite, and hence the ideals of $L$ correspond to the hereditary saturated subsets of $E^0$, by the remarks in Section~\ref{LPA-subsection}. Setting $U=\{u_{1},u_{2},u_{3},\dots\}$, $V=\{v_{1},v_{2},v_{3},\dots\}$, and $W=\{w_{1},w_{2},w_{3},\dots\}$, we see that the proper hereditary saturated subsets of $E^0$ are $\emptyset$, $U\cup V$, $W\cup V$, $V$, and $H_{n}=$ $\{v_{1},\dots, v_{n}\}$, for each $n\geq1$. It is easy to verify that $V$ is the only hereditary saturated set for which $E^{0}\backslash V$ is not downward directed. Thus, by Theorem~\ref{primeclass}, the prime ideals of $L$ are $0$, $\langle U\cup V\rangle$, $\langle W\cup V\rangle$, and $P_{n}= \langle H_{n}\rangle$, for each $n\geq1$. It follows that $(P_{n},P_{n+1})$ is a Kaplansky pair for each $n\geq1$. Since $P_{n}, P_{n+1} \subseteq \langle U\cup V\rangle$ and $P_{n}, P_{n+1} \subseteq \langle W\cup V\rangle$, it then follows that $\Spec (L)$ satisfies (KAP). However, the union of the (well-ordered ascending) chain
\[P_{1}\subseteq P_{2}\subseteq P_{3}\subseteq \cdots\]
of prime ideals of $L$ is the ideal $\langle V\rangle$, which is not prime.
\end{example}

\begin{remark}\label{well-ord-remark}
In contrast to Example \ref{row-finite Union prime}, we note that if $E$ is row-finite and if an ideal $P$ of $L_{K}(E)$ is the union of a well-ordered strictly ascending chain
\[P_{1}\subsetneq\cdots \subsetneq P_{\lambda}\subsetneq
P_{\lambda+1}\subsetneq\cdots \qquad (\lambda<\kappa)\]
of prime ideals in $L_{K}(E),$ where the ordinal $\kappa$ has uncountable cofinality, then $P$ is a prime ideal. To justify this, by Theorems~\ref{non-graded-threorem} and~\ref{primeclass}, it suffices to show that $E^{0}\backslash H$ is downward directed, where $H = P\cap E^0$. This is equivalent to showing that $T(u)\cap T(v)\nsubseteq H$ for all $u,v\in E^{0}\backslash H$ (where, for any vertex $u$, $T(u)=\{w\in E^{0}:u\geq w\}$ is the \emph{tree} of $u$). Suppose, on the contrary, that $T(u)\cap T(v)\subseteq H$. Now the set $T(u)\cap T(v)$ is countable, as $E$ is row-finite, and since $\kappa$ has uncountable cofinality, there exists a $\lambda<\kappa$ such that $T(u)\cap T(v)\subseteq H_{\lambda}$, where $H_{\lambda} = P_{\lambda} \cap E^0$. This is a contradiction, since $P_{\lambda}$ is a prime ideal and so $u,v\in E^{0}\backslash H_{\lambda}$ implies that $T(u)\cap T(v)\nsubseteq H_{\lambda}$, again by Theorem~\ref{primeclass}. Hence $P$ must be a prime ideal.
\end{remark}

\begin{remark}
The claim in Remark~\ref{well-ord-remark} remains true if we replace ``row-finite" by ``row-countable", as $T(u)\cap T(v)$ would then still be at most countable, and the same argument applies.
\end{remark}

We next explore additional sufficient conditions for (KAP) to hold for prime ideals of $L_K(E)$. We begin with an interesting property of prime ideals.

\begin{proposition} \label{Property of primes} 
Let $E$ be a graph, $K$ a field, and $L:=L_{K}(E)$. If $P \subseteq A$ are ideals of $L$, where $P$ is prime, then either $P=A$ or $P\subseteq \gr(A)$.
\end{proposition}

\begin{proof}
If $P$ is  graded, then clearly the given condition holds. Let us therefore assume that $P$ is a non-graded prime ideal, and suppose that there is an ideal $A$ such that $P\subseteq A$ and $P\nsubseteq \gr(A)$. In particular, $A$ is a non-graded ideal, and hence, by Theorem~\ref{non-graded-threorem}, $A=I(H',S')+\sum_{i\in X} \langle f_{i}(c_{i}) \rangle$, with $H'=A\cap E^{0}$ and $S'\subseteq B_{H'}$, where each $c_{i}$ is a cycle without exits in $E^{0}\backslash H'$, and each $f_{i}(x)\in K[x]$ is a polynomial of smallest degree such that $f_{i}(c_{i})\in A$. By Theorem~\ref{primeclass} and Remark~\ref{ideals in K[x,x^-1]}, we have $P=I(H,B_{H})+\langle p(c)\rangle$, where $H=P\cap E^{0}$, $c$ is a cycle without exits, based at a vertex $u$, in $E\backslash (H,B_{H})$, $v\geq u$ for all $v\in E^{0}\backslash H$, and $p(x) \in K[x]$ is an irreducible polynomial with nonzero constant term. Clearly, $c$ is the only cycle without exits in $E^{0}\backslash H$. As $P\nsubseteq \gr(A)=I(H', S')$, we see that $p(c)\notin I(H',S')$ and $c$ is also a cycle without exits in $E^{0}\backslash H'$. Since $v\geq u$ for all $v\in E^{0}\backslash H$ and hence also all $v\in E^{0}\backslash H'$, we conclude that $c_{i}=c$ for all $i \in X$. Moreover $H'=H$, because if there were a $v\in H'\backslash H$, then $v\geq u$ would imply that $u\in H'$, contradicting $p(c) \notin I(H',S')$. It follows also that $S'=B_{H}$. Thus $A$ is of the form $A = I(H,B_{H}) + \langle f(c)\rangle$, where $f(x)\in K[x]$ has nonzero constant term. Clearly, $\langle p(c)\rangle \subseteq\langle f(c)\rangle$ in $L/I(H,B_{H})$. Since  $f(x)$ was chosen to be  a polynomial of smallest degree in $K[x]$ such that $f(c)\in A$, and since $p(x)$ is irreducible in $K[x]$, it follows that $\langle p(c)\rangle =\langle f(c)\rangle$ in $L/I(H,B_{H})$. Therefore $A=I(H,B_{H})+\langle p(c)\rangle =P$, as desired.
\end{proof}

\begin{proposition} \label{LinOrderedPrimes} 
Let $E$ be a graph, $K$ a field, and $L:=L_{K}(E)$. If $\, \Spec(L)$ is totally ordered by set inclusion, then {\rm (KAP)} holds for prime ideals in $L$.
\end{proposition}

\begin{proof}
If the prime ideals of $L$ form a chain, then every semiprime ideal of $L$ is prime, by Proposition~\ref{semiprimePrime}. Thus, by Proposition~\ref{semiprimeLattice}, $\Spec(L)$ is closed under taking unions of chains. Applying Corollary~\ref{KAP for Primes}, we conclude that (KAP) holds for prime ideals in $L$.
\end{proof}

Proposition~\ref{LinOrderedPrimes} and~\cite[Theorem 27]{AAMS} can be used to completely characterize totally ordered sets which can be realized as the prime spectra of Leavitt path algebras. To achieve this, let us recall some notation and a result from~\cite{AAMS}.

\begin{definition}
\label{glb} Let $\,(P,\leq)$ be a partially ordered set. For any $S\subseteq
P$ and $p\in P$, we write $p= \gb(S)$ in case $S$ has a greatest lower bound in $P$, and it is equal to $p$. Also let
\[\RT(P)=\{p\in P:p\neq \gb(S)\text{ for all }S\subseteq P\text{ downward directed }\]
\vspace{-.25in}
\[ \text{ without a least element}\}.\]
We view $\RT(P)$ as a partially ordered subset of $P$.
\end{definition}

\begin{notation}
\label{poset properties} We assign the indicated names in case the partially ordered set $\,(P,\leq)$ satisfies the following properties.

(GLB)\quad\ Every downward directed subset of $P$ has a greatest lower bound in
$P$.

(DC) \quad\ \ For every downward directed subset $S$ of $P$ and every $p\in\RT(P)$ satisfying 

\quad\quad\quad\quad \ $p\geq \gb(S)$, we have $p\geq s$ for some $s\in S$.

(DD) \quad\ \ For every downward directed subset $S$ of $P$ such that $\gb(S)\in P$, there exists 

\quad\quad\quad\quad \ a downward directed subset $T$ of $\RT(P)$ satisfying $\gb(S)=\gb(T)$.

(KAP) \quad For all $p,q\in P$ such that $p<q$, there exist $p', q' \in P$ such that $p\leq p'<q'\leq q$ 

\quad\quad\quad\quad \ and there is no $t\in P$ satisfying $p'<t<q'$.
\end{notation}

\begin{proposition}[Proposition 38 in \cite{AAMS}] \label{Prop 38 AAMS} 
Let $\,(P,\leq)$ be a partially ordered set satisfying {\rm (GLB)}, and suppose that for every downward directed $S\subseteq P$ and every $p\in P$ satisfying $p>\gb(S)$, we have $p\geq s$ for some $s\in S$. Then $P$ satisfies {\rm (KAP)} if and only if $P$ satisfies {\rm (DD)}.
\end{proposition}

\begin{theorem} \label{Realization Thm} 
A totally ordered set $(P, \leq)$ can be realized as the prime spectrum of a Leavitt path algebra if and only if $P$ satisfies {\rm (GLB)} and {\rm (KAP)}.
\end{theorem}

\begin{proof}
\cite[Theorem 27]{AAMS} states, among other things, that a partially ordered set that satisfies (GLB), (DC), and (DD) can be realized as the prime spectrum of a Leavitt path algebra. Now, since $P$ is totally ordered, it necessarily satisfies (DC) and the condition in Proposition~\ref{Prop 38 AAMS}. Thus, if $P$ satisfies (GLB) and (KAP), then it also satisfies (DC) and (DD), and hence it can be realized as the prime spectrum of a Leavitt path algebra, by~\cite[Theorem 27]{AAMS}.

Conversely, if $P$ is order-isomorphic to the prime spectrum of a Leavitt path
algebra, then it satisfies (KAP), by Proposition~\ref{LinOrderedPrimes}.
Moreover, the prime spectrum of any ring  satisfies (GLB) by Proposition \ref{IntersectionProperty}, from which the result follows.
\end{proof}

\begin{remark} Given a field $K$ and a graph $E$, it can be shown that the ideal extension $\mathcal{E}(K, L)$ of $K$ by $L:=L_K(E)$ is a unital ring, such that there is a one-to-one inclusion-preserving correspondence between the prime ideals of $L$ and the prime ideals of $\mathcal{E}(K, L)$, except for an extra maximal ideal in the latter. Thus Theorem~\ref{Realization Thm} implies that any totally ordered set with a maximal element, that satisfies (GLB) and (KAP), can be realized as the prime spectrum of a unital ring.
\end{remark}

\section{Union-Prime Ideals of Leavitt Path Algebras}\label{UnionPrimesection}

Following~\cite{GRV}, we call a non-prime ideal in a ring $R$, which is the union of an ascending chain of prime ideals of $R$, a \emph{union-prime ideal}. In this section we study the union-prime ideals in Leavitt path algebras, since the absence of such ideals helps (KAP) hold for prime ideals. To measure the extent  to which union-prime ideals are prevalent in Leavitt path algebras, we investigate the $P^{\Uparrow}$-index, introduced in~\cite{GRV}, for such rings, and show that $P^{\Uparrow}(L_{K} (E))\leq|E^{0}|$ for any graph $E$. Examples are constructed of Leavitt path algebras having $P^{\Uparrow}$-index $\kappa$, for any prescribed infinite cardinal
$\kappa$.

\begin{proposition} \label{union-prime graded} 
Let $E$ be a graph, $K$ a field, and $L:=L_{K}(E)$. Then the union of any strictly ascending chain of prime ideals of $L$ is a graded ideal, and the intersection of any strictly descending chain of prime ideals of $L$ is a graded ideal. 

In particular, every union-prime ideal in $L$ is graded, and hence semiprime.
\end{proposition}

\begin{proof}
Let $\{P_i : i \in X\}$ be a strictly  ascending chain of prime ideals of $L$, where $(X, \leq)$ is a totally ordered set, and $P_i \subseteq P_j$ whenever $i \leq j$ ($i,j \in X$). Since this chain of prime ideals is assumed to be strictly ascending, for each $i \in X$ we can find some $i' \in X$ such that $P_i \subsetneq P_{i'}$. Then $P_i \subseteq \gr(P_{i'})$, by Proposition~\ref{Property of primes}, and hence $\bigcup_{i \in X} P_i = \bigcup_{i' \in X} \gr(P_{i'})$ is a graded ideal.

Proving that the intersection of a strictly descending chain of prime ideals is graded can be accomplished in an entirely analogous manner.

The final claim follows from the aforementioned fact that every graded ideal of a Leavitt path algebra is semiprime.
\end{proof}

\begin{definition}
Let $\tilde{P}=\{P_i : i \in X\}$ be a chain of prime ideals in a ring $R$. The cardinality of the set of non-prime unions of subchains of $\tilde{P}$ is called the \emph{index} of $\tilde{P}$, and is denoted by $\Ind (\tilde{P})$. The $P^{\Uparrow}$-\emph{index} of $R$, denoted by $P^{\Uparrow}(R)$, is the supremum of the set of cardinals $\Ind (\tilde{P})$, taken over all chains of primes $\tilde{P}$ in $R$.
\end{definition}

We require the following lemma in order to give a bound for the $P^{\Uparrow}$-index of an arbitrary Leavitt path algebra.

\begin{lemma} \label{Grade inbetween} 
Let $E$ be a graph, $K$ a field, and $L:=L_{K}(E)$. Also let $\Lambda$ be a chain of prime ideals of $L$, and let $\, \{P_{i}:i\in X\}$ be the set of the union-prime ideals that are unions of subchains of $\Lambda$, where $\,(X, \leq)$ is a totally ordered set such that $P_i \subseteq P_j$ whenever $i \leq j$ $(i,j \in X)$. 

Then for every $i\in X$ that is not a greatest element, there is a graded prime ideal $Q_{i}$ of $L$ such that $P_{i}\subsetneq Q_{i}\subsetneq P_{j}$ for all $j>i$ $(j \in X)$.
\end{lemma}

\begin{proof}
Given $i\in X$ that is not a greatest element, there exists $j \in X$ such that $j > i$, and hence there is a prime ideal $I \in \Lambda$ such that $P_{i}\subseteq I\subseteq P_{j}$, since otherwise we would have $P_{i}=P_{j}$. Thus $Q_i' = \bigcap \{I \in \Lambda : P_i \subseteq I\}$, being the intersection of a nonempty set of prime ideals, is a prime ideal (Proposition \ref{IntersectionProperty}). Then $Q_i = \gr(Q_i')$ is also a prime ideal, by Theorem~\ref{primeclass}, and it is graded, by definition. Clearly $Q_i \subseteq Q_i' \subseteq P_j$ for all $j > i$ ($i \in X$). Moreover, for all $I \in \Lambda$ such that $I \subseteq P_i$, we have $I \subsetneq Q_i'$, and hence $I \subseteq Q_i$, by Proposition~\ref{Property of primes}. It follows that $P_i \subseteq Q_i$. Finally, since for all $j > i$, the ideals $P_i$ and $P_j$ are not prime, we conclude that $P_{i}\subsetneq Q_{i}\subsetneq P_{j}$, as desired.
\end{proof}

\begin{proposition}
\label{P-index} Let $E$ be a graph, $K$ a field, and $L:=L_{K}(E)$. Then
$P^{\Uparrow}(L)\leq |E^{0}|$.
\end{proposition}

\begin{proof}
If $E^{0}$ is finite, then by Theorem~\ref{primeclass}, there can be no infinite chains of prime ideals in $L$, and hence there can be no union-prime ideals in $L$, giving $P^{\Uparrow}(L) = 0 \leq |E^{0}|$. Thus we may assume that $E^{0}$ is infinite. 

Let $\Lambda$ be a chain of prime ideals of $L$, and let $\{P_{i}:i\in X\}$ be the set of the union-prime ideals that are unions of subchains of $\Lambda$, where $(X, \leq)$ is a totally ordered set such that $P_i \subseteq P_j$ whenever $i \leq j$ $(i,j \in X)$. Also let
\[X'=\left\{
\begin{array}{ll}
X & \text{if } X \text{ has no greatest element}\\
X\setminus \{x\} & \mbox{if } x \in X \text{ is a greatest element}  
\end{array}\right..\]
Our goal is to define an injective map $f : X' \rightarrow E^{0}$. 

Given $i\in X'$, by Lemma~\ref{Grade inbetween}, there is a graded prime ideal $Q_{i}$ of $L$ such that $P_{i}\subsetneq Q_{i}\subsetneq P_{j}$ for all $j>i$ ($j \in X$). Also, by Proposition~\ref{union-prime graded}, each $P_l$ ($l \in X$) is a graded ideal. Hence, by the remarks in Section~\ref{LPA-subsection}, we can write $P_{i}=I(H_{i},S_{i})$ and $Q_{i}=I(H_{i}',S_{i}')$, for some hereditary saturated  $H_{i},H_{i}' \subseteq E^0$ and $S_{i}\subseteq B_{H_{i}}, S_{i}'\subseteq B_{H_{i}'}$. Since $P_{i}\subsetneq Q_{i}$, we have $H_{i} \cup S_i \subsetneq H_{i}' \cup S_i'$, by the order-preserving correspondence between admissible pairs and graded ideals discussed in Section~\ref{LPA-subsection}. Thus, we can define $f : X' \rightarrow E^{0}$ by choosing arbitrarily for each $i \in X'$ some $u_i \in (H_{i}' \cup S_i')\backslash (H_{i} \cup S_i)$ and setting $f(i)= u_i$.

To show that $f$ is injective, let $P_{i}=I(H_{i},S_{i})$ and $Q_{i}=I(H_{i}',S_{i}')$ be as above, let $j>i$ ($j \in X$), and set $P_j=I(H_j,S_j)$, where $H_j \subseteq E^0$ is hereditary saturated and $S_j\subseteq B_{H_j}$. By the aforementioned correspondence between admissible pairs and graded ideals, $H_{i}' \cup S_i' \subseteq H_j \cup S_j$, from which it follows that $u_j \neq u_i$, showing that $f : X' \rightarrow E^{0}$ is injective. Therefore
\[|X| \leq |X'| + 1 \leq |E^0| +1 = |E^0|, \]  
which implies that $P^{\Uparrow}(L)\leq |E^{0}|$.
\end{proof}

We conclude this note by constructing, for each cardinal $\kappa$, a directed
graph $E$ such that $P^{\Uparrow}(L_{K}(E))=\kappa$. To facilitate the construction, we require the following notation from \cite{AAMS}.

\begin{definition} \label{mainconstr}
Given a partially ordered set $(P, \leq)$ we define a graph $E_P$ by letting
\[E^0_P = \{v_p : p \in P\} \ \ \ \ \mbox{ and }  \ \ \ \ \ E^1_P = \{e_{p,q}^i : i \in \N,  \text{ and } p,q \in P \text{ satisfy } p>q\},\] where $s(e_{p,q}^i) = v_p$ and $r(e_{p,q}^i) = v_q$ for all $i \in \N$.
\end{definition}

\begin{example}
For every ordinal $\gamma$ define a partially ordered set $(P_\gamma, \leq_\gamma)$ by \[P_\gamma = \{p_\lambda : \lambda< \gamma\} \cup \{r_\lambda : \lambda < \gamma \text{ is an infinite limit ordinal}\},\] where for all ordinals $\lambda < \mu <\gamma$ we let $p_\lambda <_\gamma p_\mu$, $p_\lambda <_\gamma r_\mu$, $r_\lambda <_\gamma p_\mu$, and $r_\lambda <_\gamma r_\mu$, where $\mu$ is assumed to be an infinite limit ordinal wherever appropriate. For instance,  $E_{P_{\omega+2}}$ can be visualized as follows.

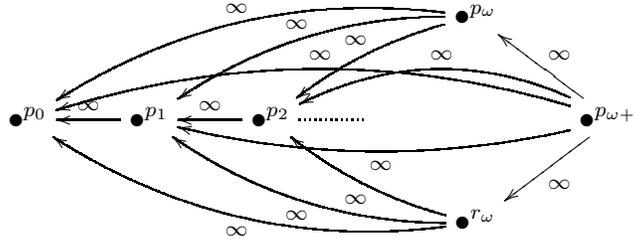
\[\xymatrix{ 
& & & & {\bullet}^{p_\omega} \ar@/^-.5pc/[lld]_{\infty} \ar@/^-1pc/[llld]_{\infty} \ar@/^-1.5pc/[lllld]_{\infty} &\\
{\bullet}^{p_0} &  {\bullet}^{p_1} \ar[l]_\infty &  {\bullet}^{p_2} \ar[l]_{\infty} & \ar@{.}[l] & &  {\bullet}^{p_{\omega +1}} \ar[dl]^{\infty} \ar[ul]_{\infty}  \ar@/^-1.6pc/[lll]_{\infty} \ar@/^1pc/[llll]^{\infty} \ar@/^-1.6pc/[lllll]_{\infty}\\
& & & & {\bullet}^{r_\omega}  \ar@/^.5pc/[llu]^{\infty} \ar@/^1pc/[lllu]^{\infty} \ar@/^1.5pc/[llllu]^{\infty} &\\
}\]
\medskip

\noindent (The symbol $\infty$ appearing adjacent to an edge indicates that there are countably infinitely many edges from one vertex to another.) We note that given two ordinals $\lambda < \gamma$, we have $P_\lambda \subseteq P_\gamma$, and $\leq_\lambda$ is the restriction of $\leq_\gamma$ to $P_\lambda$. Thus we may view $E_{P_\lambda}$ as a subgraph of $E_{P_\gamma}$ whenever $\lambda \leq \gamma$. Finally, for each ordinal $\gamma$ let $L_\gamma = L_K(E_{P_\gamma})$, for an arbitrary field $K$.

Let $\gamma$ be an ordinal. Given an ordinal $\lambda < \gamma$, we note that $E_{P_\gamma}^{0}\backslash E_{P_\lambda}^0$ is downward directed if and only if $\lambda$ is not an infinite limit ordinal, since in the case where $\lambda$ is an infinite limit ordinal $E_{P_\gamma}^{0}\backslash E_{P_\lambda}^0$ contains the two sinks $v_{p_\lambda}$ and $v_{r_\lambda}$. Noting that $E_{P_\lambda}^0$ is necessarily hereditary and saturated, we conclude, by Theorem~\ref{primeclass}, that $P_\lambda = \langle E_{P_\lambda}^0 \rangle \subseteq L_\gamma$ is a prime ideal if and only if $\lambda$ is not an infinite limit ordinal. Moreover, if $\lambda$ is an infinite limit ordinal, then $P_\lambda= \bigcup_{\alpha<\lambda} P_{\alpha}$, and hence $P_\lambda$ is a union-prime ideal. 

Now let $\kappa$ be any infinite cardinal.   Let $\gamma$  be the smallest ordinal such that the set of all infinite limit ordinals less than $\gamma$ has cardinality $\kappa$ (such exists by a straightforward set-theoretic argument).   Then $P^{\Uparrow}(L_\gamma)=\kappa$.
\end{example}

\bibliographystyle{amsalpha}

\begin{thebibliography}{A}                                                                                               

\bibitem {AAS-1}G.\ Abrams, P.\ Ara, and M.\ Siles Molina,  \textit{Leavitt Path Algebras.}  Lecture Notes in Mathematics vol. 2191, Springer-Verlag, London, 2017.  

\bibitem {AA} G.\ Abrams and G.\ Aranda Pino, \textit{The Leavitt path algebras of arbitrary graphs,} Houston J.\ Math. \textbf{34} (2008), 423--442.

\bibitem {AAMS}G.\ Abrams, G.\ Aranda Pino, Z.\ Mesyan, and C.\ Smith, \textit{Realizing posets as prime spectra of Leavitt path algebras}, J.\ Algebra \textbf{476} (2017), 267--296.

\bibitem {AAS}G.\ Abrams, G.\ Aranda Pino, and M.\ Siles Molina, \textit{Chain conditions for Leavitt path algebras}, Forum Math.\ \textbf{22} (2010), 95--114.

\bibitem {AR}G.\ Abrams and K.\ M.\ Rangaswamy, \textit{Regularity conditions for arbitrary Leavitt path algebras}, Algebr.\ Represent.\ Theory \textbf{13} (2010), 319--334.

\bibitem {AM}P.\ N.\ \'{A}nh and L.\ M\'{a}rki, \textit{Morita equivalence for rings without identity}, Tsukuba J.\ Math.\ \textbf{11} (1987), 1--16.

\bibitem {FS}J.\ W.\ Fisher and R.\ L.\ Snider, \textit{On the von Neumann regularity of rings with regular prime factor rings}, Pacific J.\ Math.\ \textbf{54} (1974), 135--144.

\bibitem {GRV}B.\ Greenfeld, L.\ Rowen, and U.\ Vishne, \textit{Unions of chains of primes}, J.\ Pure Appl.\ Alg.\ \textbf{220} (2016), 1451--1461.

\bibitem {H}M.\ Hochster, \textit{Prime ideal structure in commutative rings}, Trans.\ Amer.\ Math.\ Soc.\ \textbf{142} (1969), 43--60.

\bibitem {K-1}I.\ Kaplansky, \textit{Commutative Rings, Revised Edition}.   University of Chicago Press, Chicago and London, 1974.

\bibitem{KapOperator}  I.\ Kaplansky, \textit{Algebraic and Analytic Aspects of Operator Algebras}.  Amer.\ Math.\ Soc., Providence, R.I., 1970.   

\bibitem{Lewis} W.\ J.\ Lewis, \textit{The spectrum of a ring as a partially
ordered set,} J.\ Algebra \textbf{25} (1973), 419--434.

\bibitem{LO} W.\ J.\ Lewis and J.\ Ohm, \textit{The ordering of Spec {\rm R},} Can.\ J.\ Math.\ \textbf{28} (1976), 820--835.

\bibitem {R-1}K.\ M.\ Rangaswamy, \textit{Theory of prime ideals of Leavitt path algebras over arbitrary graphs}, J.\ Algebra \textbf{375} (2013), 73--90.

\bibitem {R-2}K.\ M.\ Rangaswamy, \textit{On generators of two-sided ideals of Leavitt path algebras over arbitrary graphs}, Comm.\ Alg.\ \textbf{42} (2014), 2859--2868.

\bibitem {S}S.\ Sarussi, \textit{Totally ordered sets and the prime spectra of rings}, Comm.\ Alg.\ \textbf{45} (2017), 411--419.

\bibitem{Speed} T.\ P.\ Speed, \textit{On the order of prime ideals,} Algebra Universalis \textbf{2} (1972), 85--87.

\bibitem {T}M.\ Tomforde, \textit{Uniqueness theorems and ideal structure for Leavitt path algebras}, J.\ Algebra \textbf{318} (2007), 270--299.
\end{thebibliography}

\end{document}